\documentclass[runningheads,a4paper]{llncs}
\usepackage{amsmath,amssymb}
\usepackage{graphicx}
\usepackage{url}
\makeatletter
\def\blfootnote{\gdef\@thefnmark{}\@footnotetext}
\makeatother

\urldef{\mailsb}\path| oguz.yayla@hacettepe.edu.tr|
\urldef{\mailsa}\path| busra.ozden@hacettepe.edu.tr|

\newcommand{\keywords}[1]{\par\addvspace\baselineskip
\noindent\keywordname\enspace\ignorespaces#1}

\usepackage[left=3cm,top=3cm,right=3cm,bottom=2cm]{geometry}

\usepackage{tikz}
 
\usetikzlibrary{arrows,automata,backgrounds,fit,shapes,decorations.pathreplacing}

\RequirePackage[T1]{fontenc}
\RequirePackage[utf8]{inputenc}

\begin{document}

\newenvironment{pf}{\noindent\textbf{Proof.}}{\hfill{$\Box$}}
 \newenvironment{ispat}{\noindent\textbf{İspat.}}{\hfill{$\Box$}}
\renewenvironment{proof}{\noindent\textit{Proof.}}{\hfill{$\Box$}}


\newcommand{\binomial}[2]{\left(\begin{array}{c}#1\\#2\end{array}\right)}
\newcommand{\zar}{{\rm zar}}
\newcommand{\an}{{\rm an}}
\newcommand{\red}{{\rm red}}
\newcommand{\codim}{{\rm codim}}
\newcommand{\rank}{{\rm rank}}
\newcommand{\Pic}{{\rm Pic}}
\newcommand{\Div}{{\rm Div}}
\newcommand{\Hom}{{\rm Hom}}
\newcommand{\im}{{\rm im}}
\newcommand{\Spec}{{\rm Spec}}
\newcommand{\sing}{{\rm sing}}
\newcommand{\reg}{{\rm reg}}
\newcommand{\Char}{{\rm char}}
\newcommand{\Tr}{{\rm Tr}}
\newcommand{\tr}{{\rm tr}}
\newcommand{\supp}{{\rm supp}}
\newcommand{\Gal}{{\rm Gal}}
\newcommand{\Min}{{\rm Min \ }}
\newcommand{\Max}{{\rm Max \ }}
\newcommand{\Span}{{\rm Span  }}

\newcommand{\Frob}{{\rm Frob}}
\newcommand{\lcm}{{\rm lcm}}

\newcommand{\ifc}{{\rm if \ }}

\newcommand{\soplus}[1]{\stackrel{#1}{\oplus}}
\newcommand{\dlog}{{\rm dlog}\,}
\newcommand{\limdir}[1]{{\displaystyle{\mathop{\rm
lim}_{\buildrel\longrightarrow\over{#1}}}}\,}
\newcommand{\liminv}[1]{{\displaystyle{\mathop{\rm
lim}_{\buildrel\longleftarrow\over{#1}}}}\,}
\newcommand{\boxtensor}{{\Box\kern-9.03pt\raise1.42pt\hbox{$\times$}}}
\newcommand{\sext}{\mbox{${\mathcal E}xt\,$}}
\newcommand{\shom}{\mbox{${\mathcal H}om\,$}}
\newcommand{\coker}{{\rm coker}\,}
\renewcommand{\iff}{\mbox{ $\Longleftrightarrow$ }}
\newcommand{\onto}{\mbox{$\,\>>>\hspace{-.5cm}\to\hspace{.15cm}$}}

\newcommand{\ord}{\mathrm{ord}}


\newcommand{\sA}{{\mathcal A}}
\newcommand{\sB}{{\mathcal B}}
\newcommand{\sC}{{\mathcal C}}
\newcommand{\sD}{{\mathcal D}}
\newcommand{\sE}{{\mathcal E}}
\newcommand{\sF}{{\mathcal F}}
\newcommand{\sG}{{\mathcal G}}
\newcommand{\sH}{{\mathcal H}}
\newcommand{\sI}{{\mathcal I}}
\newcommand{\sJ}{{\mathcal J}}
\newcommand{\sK}{{\mathcal K}}
\newcommand{\sL}{{\mathcal L}}
\newcommand{\sM}{{\mathcal M}}
\newcommand{\sN}{{\mathcal N}}
\newcommand{\sO}{{\mathcal O}}
\newcommand{\sP}{{\mathcal P}}
\newcommand{\sQ}{{\mathcal Q}}
\newcommand{\sR}{{\mathcal R}}
\newcommand{\sS}{{\mathcal S}}
\newcommand{\sT}{{\mathcal T}}
\newcommand{\sU}{{\mathcal U}}
\newcommand{\sV}{{\mathcal V}}
\newcommand{\sW}{{\mathcal W}}
\newcommand{\sX}{{\mathcal X}}
\newcommand{\sY}{{\mathcal Y}}
\newcommand{\sZ}{{\mathcal Z}}


\newcommand{\A}{{\mathbb A}}
\newcommand{\B}{{\mathbb B}}
\newcommand{\C}{{\mathbb C}}
\newcommand{\D}{{\mathbb D}}
\newcommand{\E}{{\mathbb E}}
\newcommand{\F}{{\mathbb{F}}}
\newcommand{\G}{{\mathbb G}}
\newcommand{\HH}{{\mathbb H}}
\newcommand{\I}{{\mathbb I}}
\newcommand{\J}{{\mathbb J}}
\newcommand{\M}{{\mathbb M}}
\newcommand{\N}{{\mathbb N}}
\renewcommand{\P}{{\mathbb P}}
\newcommand{\Q}{{\mathbb Q}}
\newcommand{\R}{{\mathbb R}}
\newcommand{\T}{{\mathbb T}}
\newcommand{\U}{{\mathbb U}}
\newcommand{\V}{{\mathbb V}}
\newcommand{\W}{{\mathbb W}}
\newcommand{\X}{{\mathbb X}}
\newcommand{\Y}{{\mathbb Y}}
\newcommand{\Z}{{\mathbb Z}}


\newcommand{\be}{\begin{eqnarray}}
\newcommand{\ee}{\end{eqnarray}}
\newcommand{\nn}{{\nonumber}}
\newcommand{\dd}{\displaystyle}
\newcommand{\ra}{\rightarrow}
\newcommand{\bigmid}[1][12]{\mathrel{\left| \rule{0pt}{#1pt}\right.}}
\newcommand{\cl}{${\rm \ell}$}
\newcommand{\clp}{${\rm \ell^\prime}$}

\newcommand{\TODO}[1]
{\par\fbox{\begin{minipage}{0.9\linewidth}\textbf{TODO:} #1\end{minipage}}\par}


\mainmatter  

\title{Almost $p$-ary Sequences}

\titlerunning{Almost $p$-ary sequences}
%
%
\author{Büşra Özden \and O\u guz Yayla\footnote{Corresponding author}}
\authorrunning{Özden, Yayla}

\institute{Department of Mathematics, Hacettepe University\\Beytepe, 06800, Ankara, Turkey\\ \mailsa\\
\mailsb}

\toctitle{Lecture Notes in Computer Science}
\tocauthor{Authors' Instructions}
\maketitle
\begin{abstract}
		In this paper we study almost $p$-ary sequences and their autocorrelation coefficients.  We first study the number $\ell$ of distinct out-of-phase autocorrelation coefficients for an almost $p$-ary sequence of period $n+s$ with $s$ consecutive zero-symbols. We prove an upper bound and a lower bound on $\ell$. It is shown that $\ell$ can not be less than $\min\{s,p,n\}$. In particular, it is shown that a nearly perfect sequence with at least two consecutive zero symbols does not exist. Next we define a new difference set, partial direct product difference set (PDPDS), and we prove the connection between an almost $p$-ary nearly perfect sequence of type $(\gamma_1, \gamma_2)$ and period $n+2$ with two consecutive zero-symbols and a cyclic $(n+2,p,n,\frac{n-\gamma_2 - 2}{p}+\gamma_2,0,\frac{n-\gamma_1 -1}{p}+\gamma_1,\frac{n-\gamma_2 - 2}{p},\frac{n-\gamma_1 -1}{p})$ PDPDS for arbitrary integers $\gamma_1$ and $\gamma_2$. Then we prove a necessary condition on $\gamma_2$ for the existence of such sequences. In particular, we show that they don't exist for $\gamma_2 \leq -3$. 	
		\blfootnote{\textup{2010} \textit{Mathematics Subject Classification}: \textup{05B10, 94A55.}}
 
	\end{abstract}
	
	\section{Introduction}
	Let $\zeta_p\in \C$ be a primitive $p$-th root of unity for some prime number $p$.
	A sequence $\underline{a} = (a_0,a_1,\ldots,a_{n-1},\ldots)$ of period $n$ with $a_i = \zeta_p ^{b_i}$ for some integer $b_i$, $i =0, 1, \ldots,n-1$ is called a \textit{$p$-ary sequence}. If $a_{i_j} = 0$ for all $j=1,2,\ldots , s$ where $\{ i_1,i_2, \ldots , i_s \} \subset \{0,1, \ldots , n-1\}$ and $a_i = \zeta_p^{b_i}$ for some integer $b_i$, $i \in  \{0,1, \ldots , n-1\} \backslash \{ i_1,i_2, \ldots , i_s \}$, then we call $\underline{a}$ an \textit{almost $p$-ary sequence with $s$ zero-symbols}. 
	For instance, $\underline{a}=(\zeta_3^3,\zeta_3^2,\zeta_3^4,\zeta_3^2,1,\ldots)$ is a $3$-ary sequence of period $5$ and  $\underline{a}=(0,\zeta_7^3,1,\zeta_7^3,0,0,\zeta_7^5,\zeta_7^6,\zeta_7^6,\zeta_7^5,,\ldots)$ is an almost $7$-ary sequence with 3 zero-symbols of period $10$.  It is widely used that a sequence with one zero-symbol is called an almost $p$-ary sequence. But in this paper we use this notation for a $p$-ary sequence with $s$ zero-symbols, for $s \geq 0$. 
	
	For a sequence $\underline{a}$ of period $n$, its \textit{autocorrelation 
		function} $C_{\underline{a}}(t)$ is defined as
	$$ C_{\underline{a}}(t) = \sum_{i=0}^{n-1}{a_i\overline{a_{i+t}}},$$ for $0\leq t \leq n-1$ where $\overline{a}$ is the complex conjugate of $a$.  The values $C_{\underline{a}}(t)$ at $1 \leq t \leq n-1$ are called \textit{the out-of-phase autocorrelation coefficients} of $\underline{a}$. Note that the autocorrelation function of $\underline{a}$ is periodic with $n$.
	
	We call an almost $p$-ary sequence $\underline{a}$ of period $n$ a \textit{nearly perfect
		sequence} (NPS) of type $(\gamma_1, \gamma_2)$ if all out-of-phase autocorrelation coefficients of $\underline{a}$ are either $\gamma_1$ or $\gamma_2$. We write \textit{NPS of type $\gamma$} to denote an NPS of type $(\gamma,\gamma)$. Moreover, a sequence is called \textit{perfect sequence} (PS) if it is an NPS of type $(0,0)$. We also note that there is another notion of \textit{almost perfect sequences} which is a $p$-ary sequence $\underline{a}$ of period $n$ having $C_{\underline{a}}(t)=0$ for all $1\leq t \leq n-1$ -with exactly one exception \cite{jungnickel1999perfect}.
	
	There are some applications of almost $p$-ary NPS of type $\gamma$ \cite{golomb2005signal,jungnickel1999perfect}. Therefore, nearly perfect sequences have been studied by many authors. Jungnickel and Pott \cite{jungnickel1999perfect} studied binary NPS of type $|\gamma|\leq2$. Ma and Ng \cite{ma2009non} obtained a relation between a $p$-ary NPS of type $|\gamma|\leq 1$ and a direct product difference set (DPDS) and obtained nonexistence on some $p$-ary NPS of type $|\gamma|\leq 1$ by using character theory. Later Chee et al. \cite{chee2010almost} extended the methods due to Ma and Ng \cite{ma2009non} to almost $p$-ary NPS of types $\gamma=0$ and $\gamma=-1$ with one zero-symbol. Then, Özbudak et al. \cite{ozbudak2012nonexistence} proved the nonexistence of almost $p$-ary NPS with one zero-symbol at certain values. Liu and Feng \cite{liu2016new} obtained new nonexistence results on  $p$-ary  PS  and related difference sets (RDS) by using some results on cyclotomic fields and their sub-fields. They also considered almost $p$-ary PS with one zero-symbol. Chang Lv \cite{lv2017non} obtained nonexistence of almost $p$-ary PS with $s\leq 1$ zero-symbol for $p\equiv 5 \mod 8$ (resp.~$ p\equiv 3\mod 4$) and period $p^a qn'$ (resp.~$p^a q^ln'$) by considering equations cyclotomic fields satisfied by perfect sequences. Niu et al.~\cite{niu2018non} studied a binary sequence of the periodical with a 2-level autocorrelation value, this solves three open problems given by Jungnickel and Pott \cite{jungnickel1999perfect}. Moreover, second author \cite{yayla2016nearly} proved an equality between a $p$-ary NPS of type $\gamma$  and a DPDS for an arbitrary integer $ \gamma $. In addition, he extended this result for an almost $p$-ary NPS with one zero-symbol, and  proved its nonexistence cases by self-conjugacy condition. Then,  a non existence result was proven for an almost $p$-ary NPS with $s\geq 2$ zero-symbols.
	
	In this paper, we study almost $p$-ary sequences and their some properties. We first prove some bounds on the number of  distinct out-of-phase autocorrelation coefficients of an almost $p$-ary  sequence of period $n+s$ with $s$ consecutive zero-symbols (see Theorem \ref{th:nbrg}). In particular, we prove that the number of distinct out-of-phase autocorrelation coefficients can not be less than $\min\{s,p,n\}$. Next, we define a new difference set called partial direct product difference set (see Definition \ref{df:pDPDS}) and prove that a $p$-ary NPS of type $(\gamma_1, \gamma_2)$ is equivalent to a PDPDS (see Theorem \ref{th:p-DPDS}). And, we show that they exist only if $p$ divides $n-\gamma_2-2$ and $n-\gamma_1-1$. Finally, we show a bound on $\gamma_2$  for the existence of an almost $p$-ary sequence of type $(\gamma_1, \gamma_2)$ with two consecutive zero-symbols (see Theorem \ref{thm:bound}). As a consequence of this result we show that such sequences don't exist if $\gamma_2 \leq -3$ (see Corollary \ref{cor:gamma2}).
	
	This paper is organized as follows. In Section \ref{sec.pre} some preliminary results are presented. Then we present some properties of the autocorrelation coefficients of an almost $p$-ary sequence in Section \ref{sec:autocor}.
	Then, we study the relation between an almost $p$-ary NPS and a partial direct product difference set in Section \ref{sec:pDPDS}. 
	
	\section{Preliminaries}
	\label{sec.pre}
	We first give the definition of a direct product difference set \cite{ma2009non}.
	
	\begin{definition}
		Let $G = H \times N$, where the order of $H$ and $N$ are $n$ and $m$. A subset $R$ of $G$, $|R| =k$, is called an $(n,m,k,\lambda_1,\lambda_2,\mu)$ direct product difference set (DPDS) in $G$ relative to $H$ and $N$ if 
		differences $r_1r_2^{-1}$, $r_1, r_2 \in R$ with $r_1 \neq r_2$ represent
		\begin{itemize}
			\item all non identity elements of $H$ exactly $\lambda_1$ times,
			\item all non identity elements of $N$ exactly $\lambda_2$ times,
			\item all non identity elements of $G \backslash H \cup N$ exactly $\mu$ times.
		\end{itemize}
	\end{definition}
	We can also define a difference set by using the group-ring algebra notation. 
	Let $\sum_{g \in R}{g} \in \Z[G]$ be an element of the group ring $\Z[G]$, for simplicity we will denote the sum by $R$. If $R$ is an $(m,n,k,\lambda_1,\lambda_2,\mu)$-DPDS in $G$ relative to $H$ and $N$ then
	\be \label{eqn:DPDS}
	RR^{(-1)} = (k - \lambda_1 - \lambda_2 + \mu) + (\lambda_1 - \mu)H + (\lambda_2 - \mu)N + \mu G
	\ee
	holds in $\Z[G]$.

	The following result on vanishing sums of roots of unity due to Lam and Leung \cite{LL2000}, see also \cite[Proposition 2.1]{winterhof2014non}.
	\begin{lemma} \cite{LL2000} \label{thm.LL} 
		Let $m$ be an integer with prime factorization $m = p_1^{a_1}p_2^{a_2}\ldots p_\ell^{a_\ell}$. If there are  
		$m$-th roots of unity $\xi_1,\xi_2,\ldots,\xi_v$ with $\xi_1+\xi_2+\ldots+\xi_v = 0$, then $v = p_1t_1+p_2t_2+\ldots+p_\ell t_\ell$ with non-negative integers $t_1,t_2,\ldots,t_\ell$.
	\end{lemma}

	We now give the relation between an NPS and a DPDS. Let p be a prime, $n \geq 2$ be an integer, and $\underline{a} = (a_0,a_1,\ldots,a_{n},\ldots)$ be an almost $p$-ary sequence of period $n+s$ with $s$ zero-symbol such that $a_{i_j} = 0$ for all $j=1,2,\ldots , s$ where $\{ i_1,i_2, \ldots , i_s \} \subset \{0,1, \ldots , n+s-1\}$.
	Let $H = \langle h \rangle$ and $P = \langle g \rangle $ be the (multiplicatively written) cyclic groups of order $n+s$ and $p$. Let $G$ be the group defined as $G=H \times P$.
	We choose a primitive $p$-th root of 1, $\zeta_p \in \C$. For $i \in  \{0,1, \ldots , n+s-1\} \backslash \{ i_1,i_2, \ldots , i_s \}$ let $b_i$ be the integer in $\{0,1,2,\ldots,p-1\}$ such that $a_i = \zeta_p^{b_i}$. 
	Let $R_a$ be the subset of $G$ defined as
	\be \label{eqn:NPS2DPDS}
	R_a = \{(g^{b_i}h^i) \in G : i \in  \{0,1, \ldots , n+s-1\} \backslash \{ i_1,i_2, \ldots , i_s \}\}.
	\ee
	
	In the following we present a known result between an almost $p$-ary sequence of type $\gamma$ with one zero-symbol  and a DPDS for an integer $\gamma$.  
	
	\begin{theorem}\cite{yayla2016nearly} \label{th:DPDS:almost}
		$\underline{a}$ is an almost $p$-ary NPS of period $n+1$ and type $\gamma$ with one zero-symbol if and only if $R_a$ defined as in \eqref{eqn:NPS2DPDS} is an $(n+1,p,n,\frac{n-\gamma-1}{p}+\gamma,0,\frac{n-\gamma-1}{p})$-DPDS in $G$ relative to $H$ and $P$. In particular, $p$ divides $n-\gamma-1$.
	\end{theorem}
	
	\section{Autocorrelation Coefficients}
	\label{sec:autocor}
	
	In this section we use the notation of the previous section. We first give an extension of a well known divisibility result on difference sets whose proof follows by counting the number of elements in the difference table.

	\begin{proposition} \label{prop:counting}
		Let $\underline{a}$ is a sequence of period $n+s$ with $s$  zero-symbol. Let $R_a$ be a $(n+s,p,n,\lambda_1,\lambda_2,\mu)$-DPDS. Then  $(n+s-1)(\lambda_1 + \mu(p-1)) = n^2-n$.
	\end{proposition}
	\begin{proof}
		We know that there are $n^2-n$ nonidentity elements in the difference table of $R_a$. They correspond to $\lambda_1$ times nonidentity elements of $\Z_{n+s}\times \{0\} $ and $\mu$ times nonidentity elements of $\Z_{n+s}\times \Z_p \backslash (\Z_{n+s}\times \{0\} \cup \{0\} \times \Z_{p})$.   Hence the result follows.
	\end{proof}\\
	
	It is now clear that $R_a$ is not a DPDS if $(n+s-1) \nmid n^2-n$.
	If the zero-symbols are  consecutive we have more than the divisibility condition. The proof of the following result follows similarly.
	\begin{proposition} \label{prop:cons}
		Let $\underline{a}$ is a sequence of period $n+s$ with $s \geq 1$ consecutive zero-symbol. Let $R_a$ be a $(n+s,p,n,\lambda_1,\lambda_2,\mu)$-DPDS. Then  $n-i=\lambda_1+\mu(p-1)$ for $i=1,2,\ldots,s$, and so $s=1$. In addition, if $s=0$ then $n=\lambda_1+\mu(p-1)$ holds.
	\end{proposition}
	\begin{proof}
		If $R$ is a $(n+s,p,n,\lambda_1,\lambda_2,\mu)$-DPDS, then by checking the sub-diagonal entries in difference table we have $n-i=\lambda_1+\mu(p-1)$ for i=1,2,\ldots,s. Thus, right hand side of this equation is fixed and so this can only hold for one index $i$, i.e.~$s=1$. The later statement of the theorem holds similarly.
	\end{proof}
	
	\begin{example}
		Let $\underline{a}=(0,\zeta_3^2,\zeta_3^2,\zeta_3^2,1,\zeta_3^2,\zeta_3, \zeta_3,\zeta_3^2,1,\zeta_3^2,\zeta_3^2,\zeta_3^2)$ be a $3$-ary NPS of period $13$ and $s=1$. Here we have $\lambda_1=5$ and $\mu=3$, then Proposition \ref{prop:cons} is satisfied. Similarly, let $\underline{a}=(\zeta_3^2,\zeta_3^2,\zeta_3^2,\zeta_3^2,1)$ be a $3$-ary NPS of period $5$ and $s=0$ In this case we have $\lambda_1=3$ and $\mu=1$, then Proposition \ref{prop:cons}  is also satisfied. 
	\end{example}

	Now we show that if $s \geq 2$ there does not exist a nearly perfect sequence with only one out-of-phase autocorrelation coefficient. Before that we give the following lemma. 
	
	\begin{lemma}\label{lem:equiv} 
		The number of congruence classes in a set $\{1,2,\ldots,s\}$ modulo some prime $p \leq s$ is equivalent to $p$.
	\end{lemma}
	
	\begin{theorem} \label{th:nbrg}
		Let $\underline{a}$ be an almost $p$-ary sequence of period $n+s$ with $s$ consecutive zero-symbols. Let $\ell$ be the number of distinct elements in the set $\{C_{\underline{a}}(1), C_{\underline{a}}(2), \ldots, C_{\underline{a}}(n+s-1) \}$. Then, $\min\{s,p,n\} \leq \ell \leq n-1+\min\{n,s\} $.
	\end{theorem}
	
	\begin{proof}
		We first consider the case $n>s$. Let $B=\{C_{\underline{a}}(1),C_{\underline{a}}(2),\ldots,C_{\underline{a}}(s),\ldots,\linebreak C_{\underline{a}}(n),\ldots ,C_{\underline{a}}(n+s-2),C_{\underline{a}}(n+s-1) \}$. Then we have
		\begin{equation*}
		\begin{aligned}
		B=\{ a_s\overline{a}_{s+1}+a_{s+1}\overline{a}_{s+2}+\ldots +a_{n+s-2}\overline{a}_{n+s-1}, \\
		a_s\overline{a}_{s+2}+a_{s+1}\overline{a}_{s+3}+\ldots +a_{n+s-3}\overline{a}_{n+s-1},\\
		\ldots \\ 
		a_s\overline{a}_{2s}+a_{s+1}\overline{a}_{2s+1}+\ldots +a_{n-1}\overline{a}_{n+1}, \\
		\ldots \\ 
		a_{2s}\overline{a}_{s}+a_{2s+1}\overline{a}_{s+1}+\ldots +a_{n+1}\overline{a}_{n-1}, \\
		\ldots \\ 
		a_{s+2}\overline{a}_s+a_{s+3}\overline{a}_{s+1}+\ldots +a_{n+s-1}\overline{a}_{n+s-3}, \\
		a_{s+1}\overline{a}_{s}+a_{s+2}\overline{a}_{s+1}+\ldots +a_{n+s-1}\overline{a}_{n+s-2} \}.
		\end{aligned}
		\end{equation*}
	    And let $\ell$ be the number of distinct elements in $B$. If all the elements in $B$ are distinct, then maximum value of $\ell$ is $\ell_{max}=\#B=n+s-1$. On the other hand, $C_{\underline{a}}(i)$ is the sum of $n-i$ elements for $i=1,2,\ldots, s$; $C_{\underline{a}}(i)$ is the sum of $n-s$ elements for $i=s+1,s+2,\ldots, n$ and $C_{\underline{a}}(i)$ is the sum of $i-s$ elements for $i=n+1,n+2,\ldots, n+s-1$. We note that the number of summands in $C_{\underline{a}}(i)$ equals to the number of summands in $C_{\underline{a}}(n+s-i)$ for $i=1,2,\ldots,s$ and the number of summands in $C_{\underline{a}}(i)$ equals to the number of summands in $C_{\underline{a}}(s)$ for $i=s+1,s+2,\ldots, n$. Thus, we can decide the maximum value of $\ell$ by checking the equality of $C_{\underline{a}}(t)$ values for $t\in\{1,2,\ldots ,s\}$. 
		If $C_{\underline{a}}(1)=C_{\underline{a}}(2)$ then $C_{\underline{a}}(1)-C_{\underline{a}}(2)=0$, that is $$a_s\overline{a}_{s+1}+a_{s+1}\overline{a}_{s+2}+\ldots +a_{n+s-2}\overline{a}_{n+s-1}-(a_s\overline{a}_{s+2}+\ldots +a_{n+s-3}\overline{a}_{n+s-1})=0.$$ 
		Then we have, 
		$$\zeta_p ^{b_s-b_{s+1}}+\zeta_p ^{b_{s+1}-b_{s+2}}+\ldots +\zeta_p ^{b_{n+s-2}-b_{n+s-1}}-(\zeta_p ^{b_s-b_{s+2}}+\ldots +\zeta_p ^{b_{n+s-3}-b_{n+s-1}})=0$$ 
		where $a_i = \zeta_p^{b_i}$ for some integer $b_i$, so 
		\be\nn
		\begin{array}{lr}
			\zeta_p ^{b_s-b_{s+1}}+\zeta_p ^{b_{s+1}-b_{s+2}}+\ldots +\zeta_p ^{b_{n+s-2}-b_{n+s-1}}+\\
			(\zeta_p ^{p-1}+\zeta_p ^{p-2}+\ldots+\zeta_p)(\zeta_p ^{b_s-b_{s+2}}+\ldots +\zeta_p ^{b_{n+s-3}-b_{n+s-1}})=0.
		\end{array} 
		\ee
		Hence we get that $n-1+(p-1)(n-2)=p(n-2)+1$ number of $p$-th root of unities sum up to zero. 
		Similarly,  the number of $p$-th root of unities in difference $$C_{\underline{a}}(i)-C_{\underline{a}}(j)$$ is $p(n-j)+j-i$ for $j>i$. By Lemma \ref{thm.LL}, the above equation vanishes only if $p|j-i$, that is $i \equiv j \mod p$. 
		If $s>p$ then we have at least $p$ distinct equivalence classes in the set  $\{1,2,\ldots,s\}$ modulo $p$ by Lemma \ref{lem:equiv}, and so $\ell _{min} =p$. If $p\geq s$ then $p \nmid j-i$ for distinct $i,j \in \{1,2,\ldots,s\}$, and so we get $\ell _{min} = s$. 
		
		Similarly, in the case of $s\geq n$, $\ell_{max}=n-1+1+n-1=2n-1$, $\ell _{min}=p$ for $n>p$ and $\ell _{min}=n$ for $p\geq n$.
	\end{proof}
	
	\begin{example}
		For almost 3-ary sequences $\underline{a_1}=(0,0,\zeta_3,\zeta_3,\zeta_3,\zeta_3)$, $\underline{a_2}=(0,0,\zeta_3^2,\linebreak\zeta_3,\zeta_3,\zeta_3^2)$, $\underline{a_3}=(0,0,\zeta_3,1,\zeta_3,\zeta_3)$ and $\underline{a_4}=(0,0,\zeta_3^2,\zeta_3^2,1,1)$, the number of distinct autocorrelation coefficients satisfies $\ell=2,3,4,5$ respectively. Here we have $s=2$, $n=4$, $p=3$. So Theorem \ref{th:nbrg} is satisfied, i.e.~$min\{2,4,3\}\leq\ell\leq 4-1+min\{2,4\}$.
	\end{example}

	\begin{example} Almost 3-ary sequences $\underline{a}_1=(1,0,0,1,0,1,1)$ , $\underline{a}_2=(\zeta_3,0,0,\zeta_3,0,\linebreak\zeta_3,\zeta_3)$ and $\underline{a}_3=(\zeta_3^2,0,0,\zeta_3^2,0,\zeta_3^2,\zeta_3^2)$ are NPS of type (2,2) and period 7 with 3 zero-symbols. Hence, Theorem \ref{th:nbrg} does not hold for almost $p$-ary sequences with non consecutive zero-symbols. \end{example} 
	
	Now we give a direct consequence of Theorem \ref{th:nbrg}, which says that one can not get an NPS of type $\gamma$ by adding extra zero-symbols at consecutive positions. 
	
	\begin{corollary} \label{crl:ntNPS}
		For $n\in \Z^{+}$, a prime number $p$, $s\geq 2$ and $\gamma \in Z$, there does not exist an almost $p$-ary NPS of type $\gamma$ and period $n+s$ with $s$ consecutive zero-symbols.
	\end{corollary}
	
	\section{Partial Direct Product Difference Sets}
	\label{sec:pDPDS}
	
	Here we give a new difference set definition, called partial direct product difference set (PDPDS).
	
	\begin{definition}\label{df:pDPDS}
		Let $G = H \times P$, where the order of $H= \langle h \rangle$ and $P= \langle g \rangle $ are $n$ and $m$. A subset $R$ of $G$, $|R| =k$, is called an $(n,m,k,\lambda_1,\lambda_2,\lambda_3,\mu_1,\mu_2)$ partial direct product difference set (PDPDS) in $G$ relative to $H$ and $P$ if 
		differences $r_1r_2^{-1}$, $r_1, r_2 \in R$ with $r_1 \neq r_2$ represent
		\begin{itemize}
			\item all elements of $\{h^2,h^3,\ldots,h^n\}$ exactly $\lambda_1$ times,
			\item all non identity elements of $P$ exactly $\lambda_2$ times,
			\item all elements of $\{h\ ,h^{n+1}\}$ exactly $\lambda_3$ times,
			\item all elements of $\{h^2,h^3,\ldots,h^n\} \times \{g,g^2,\ldots,g^{p-1}\}$ exactly $\mu_1$ times,
			\item all non identity elements of $\{h\ ,h^{n+1}\} \times \{g,g^2,\ldots,g^{p-1}\}$ exactly $\mu_2$ times.
		\end{itemize}
	\end{definition}
	In the group-ring algebra notation, if $R$ is an $(n,m,k,\lambda_1,\lambda_2,\lambda_3,\mu_1,\mu_2)$-PDPDS in $G$ relative to $H$ and $P$ then
	\be\label{eqn:p-DPDS}
	\begin{aligned}
		RR^{(-1)} &= (k - \lambda_1 - \lambda_2 + \mu_1) + (\lambda_1 - \mu_1)H + (\lambda_2 - \mu_1)P + \mu_1 G+\\
		&(\lambda_3 - \lambda_1)\{h,h^{n+1}\}+ (\mu_2-\mu_1)(\{h\ ,h^{n+1}\} \times \{g,g^2,\ldots,g^{p-1}\})
	\end{aligned}
	\ee			
	holds in $\Z[G]$.
	
	\begin{remark} Let $G = H \times P$, $H= \langle h \rangle$ and $P= \langle g \rangle $ of order $n$ and $m$. An $(n,m,k,\lambda_1,\lambda_2,\lambda_1,\mu,\mu)$-PDPDS in $G$ relative to $H$ and $P$ is an $(n,m,k,\lambda_1,\lambda_2,\mu)$-DPDS in $G$ relative to $H$ and $P$. Moreover,  an $(n,m,k,\lambda,0,\linebreak\lambda,\lambda,\lambda)$-PDPDS in $G$ relative to $H$ and $P$ is an $(n,m,k,\lambda)$-RDS in $G$ relative to $P$. Finally, an $(n,m,k,\lambda,\lambda,\lambda,\lambda,\lambda)$-PDPDS in $G$ relative to $H$ and $P$ is an $(nm,k,\lambda)$-DS in $G$.
	\end{remark}
	
	We extend Proposition \ref{prop:counting} for DPDS to PDPDS below. 
	\begin{proposition}\label{pDPDS equality}
		Let $\underline{a} = (a_0,a_1,\ldots,a_{n+1},\ldots)$ be a almost $p$-ary sequence of period $n+2$ such that $a_0=0$ and $a_1=0$. Let $R$ be $(n+2,m,k,\lambda_1,\lambda_2,\lambda_1,\mu_1,\mu_2)$-PDPDS. Then $(n-1)(\lambda_1+(p-1)\mu_1)+2(\lambda_3+(p-1)\mu_2)=n^2-n$.
	\end{proposition}
	
	\begin{proof}
		We know that there are $n^2-n$ non-identity elements in the difference table of $R$ and we know that $\lambda_1$ times $\{h^2,h^3,\ldots,h^n\}\times \{0\}$, $\lambda_3$ times $\{h\ ,h^{n+1}\}\times \{0\}$, $\mu_1$ times $\{h^2,h^3,\ldots,h^n\} \times \{g,g^2,\ldots,g^{p-1}\}$ and $\mu_2$ times  $\{h\ ,h^{n+1}\} \times \{g,g^2,\ldots,g^{p-1}\}$ from the definition of PDPDS. So, $(n-1)\lambda_1+2\lambda_3+(n-1)(p-1)\mu_1+2(p-1)\mu_2=n^2-n$. Hence the result follows.
	\end{proof}\\
	
	Now we present a relation between a PDPDS and an NPS with two distinct out-of-phase autocorrelation coefficients.
	\begin{theorem} \label{th:p-DPDS}
		Let $p$ be a prime, $n \geq 2$ be an integer, and $\underline{a} = (a_0,a_1,\ldots,a_{n+1},\ldots)$ be a almost $p$-ary sequence of period $n+2$ such that $a_0 = 0$ and $a_1 =0$.
		Let $H = \langle h \rangle$ and $P = \langle g \rangle $ be the (multiplicatively written) cyclic groups of order $n+2$ and $p$, respectively. Let $G$ be the group defined as $G=H \times P$.
		We choose a primitive $p$-th root of 1, $\zeta_p \in \C$. For $2 \leq i \leq n+1$ let $b_i$ be the integer in $\{0,1,2,\ldots,p-1\}$ such that $a_i = \zeta_p^{b_i}$.  
		Let $R$ be the subset of $G$ defined as
		\be \nn
		R = \{(g^{b_i}h^i) \in G : 2 \leq i \leq n+1\}.
		\ee
		Then
		$\underline{a}$ is an almost $p$-ary NPS of type $(\gamma_1, \gamma_2)$ if and only if $R$ is an $(n+2,p,n,\frac{n-\gamma_2 - 2}{p}+\gamma_2,0,\frac{n-\gamma_1 -1}{p}+\gamma_1,\frac{n-\gamma_2 - 2}{p},\frac{n-\gamma_1 -1}{p})$ PDPDS in $G$ relative to $H$ and $N$. 
	\end{theorem}
	\begin{proof}
		Let $A = \sum_{i = 0}^{n-1}{a_ih^i} \in \C[H]$. Then we have
		$$A\overline{A}^{(-1)} = \sum_{t = 0}^{n-1}{C_a(t)h^t}.$$
		Let $\chi$ be a character on $P$. We extend $\chi$ to $G$ such that $\chi(h) = h$. Let $\sigma \in \Gal(\Q(\zeta_p)\backslash \Q)$ such that $\sigma(\zeta_p) = \chi(\zeta_p)$.
		If $\chi$ is a nonprincipal character on $P$, then we have $\chi(R) = A^{\sigma}$, and so
		$$\chi(RR^{(-1)}) = (A\overline{A}^{(-1)})^{\sigma}.$$
		On the other hand, if $\chi$ is a principal character on $P$, then we have
		$$\chi(R) = H-\{1,h\}$$ and also $$\chi(R^{(-1)}) = H-\{1,h^{n+1}\}.$$
		Then
		\be \nn
		\chi(RR^{(-1)}) = \left\lbrace \begin{array}{ll}  (H-\{1,h\})( H-\{1,h^{n+1}\})& \mbox{if } \chi \mbox{ is principal on }P,\\ \sum_{t = 0}^{n-1}{C_a(t)^\sigma h^t} & \mbox{if } \chi \mbox{ is nonprincipal on $P$}. \end{array} \right.
		\ee
		So
		\be \nn
		\chi(RR^{(-1)}) = \left\lbrace \begin{array}{ll} (n-2)H+2+\{h\ ,h^{n+1}\} & \mbox{if } \chi \mbox{ is pr.~on }P,\\ \sum_{t = 0}^{n-1}{C_a(t)^\sigma h^t} & \mbox{if } \chi \mbox{ is nonpr.~on $P$}. \end{array} \right.
		\ee
		If $\underline{a}$ is an NPS of type $(\gamma_1,\gamma_2)$, then
		\be \label{eq:NPS-pDPDS} 
		\chi(RR^{(-1)}) = \left\lbrace \begin{array}{ll} (n-2)H+2+\{h\ ,h^{n+1}\} & \mbox{if } \chi \mbox{ is pr.~on }P,\\ n-\gamma_2+(\gamma_1-\gamma_2)\{h\ ,h^{n+1}\}+\gamma_2H & \mbox{if } \chi \mbox{ is nonpr.~on $P$}. \end{array} \right.
		\ee
		By extending $\chi$ to $H$ we obtain
		\be \label{diag:s2}
		\chi(RR^{(-1)}) = \left\lbrace
		\begin{array}{ll}
			n^2 & \mbox{if } \chi \mbox{ is pr.~on }P  \mbox{ and } H,\\
			2+h+h^{-1} & \mbox{if } \chi \mbox{ is pr.~on }P  \mbox{ and nonpr.~on }H,\\
			n+\gamma_2(n-1)+2\gamma_1 & \mbox{if } \chi \mbox{ is nonpr.~on }P   \mbox{ and pr.~on }H,\\
			n -\gamma_2+h+h^{-1} & \mbox{if } \chi \mbox{ is nonpr.~on }P  \mbox{ and nonpr.~on }H.\\
		\end{array}
		\right.
		\ee
		First, we can consider
		\be
		\begin{array}{ll}
			\label{eqn:pDPDS2}
			RR^{(-1)} = &n-x + xH -zP + zG+x'\{h\ ,h^{n+1}\}+\\
			&z'(\{h\ ,h^{n+1}\} \times \{g,g^2,\ldots,g^{p-1}\})
		\end{array} 
		\ee
		for some integers $x,x',z,z'$.
		We get the following equations by \eqref{diag:s2} and \eqref{eqn:pDPDS2}
		\be \label{eqn:eq1}
		2+h+h^{-1}=n-x+zp+x'(h+h^{-1})
		\ee
		\be \label{eqn:eq2}
		n-\gamma_2+(h+h^{-1})(\gamma_1-\gamma_2)=n-x+x'(h+h^{-1})-z'(h+h^{-1})
		\ee
		%
		If we solve (\ref{eqn:eq1}) and (\ref{eqn:eq2}) together, then we get $x=\gamma_2$, $z=\frac{n-\gamma_2 - 2}{p}$, $z'=\frac{\gamma_2-\gamma_1+1}{p}$, $x'=\gamma_1-\gamma_2+\frac{\gamma_2-\gamma_1+1}{p}$.
		Now we can easily get that $R$ is an $(n+2,p,n,\frac{n-\gamma_2 - 2}{p}+\gamma_2,0,\frac{n-\gamma_1 -1}{p}+\gamma_1,\frac{n-\gamma_2 - 2}{p},\frac{n-\gamma_1 -1}{p})$-PDPDS by using (\ref{eqn:p-DPDS}) and (\ref{eqn:pDPDS2}). 
		
		On the other hand, $(n+2,p,n,\frac{n-\gamma_2 - 2}{p}+\gamma_2,0,\frac{n-\gamma_1 -1}{p}+\gamma_1,\frac{n-\gamma_2 - 2}{p},\frac{n-\gamma_1 -1}{p})$-PDPDS satisfies the diagram \eqref{eq:NPS-pDPDS} for any character $\chi$ on $G$. So we get that $\underline{a} = (a_0,a_1,\ldots,a_{n+1})$ is an NPS of type $(\gamma_1,\gamma_2)$.
	\end{proof}\\
	
	Theorem \ref{th:p-DPDS} gives a necessary condition on the existence of a NPS with two distinct out-of-phase autocorrelation coefficients. Moreover this theorem gives bound on $\gamma_1$,$\gamma_2$. We state this condition in Corollary \ref{cor:NPS2}. After that we give an example of Theorem \ref{th:p-DPDS}.
	\begin{corollary} \label{cor:NPS2}
		If $\underline{a}$ is an almost $p$-ary NPS of type $(\gamma_1, \gamma_2)$ and length $n+2$, then $p$ divides $n-\gamma_2-2$ and $n-\gamma_1-1$. And there exists an almost $ p $-ary sequence of  type $ (\gamma_1,\gamma_2) $ and period $ n+2 $ with two consecutive zero-symbols for $-\mu_1\leq\gamma_2\leq n-2$ and $-\mu_2\leq\gamma_1\leq n-1$.
	\end{corollary}
	
	\begin{example}
		Sequence $\underline{a}=(0,0,\zeta_3,\zeta_3,\zeta_3)$ is an 
		almost 3-ary NPS of type (2,1) and $R$ is an $(3+2,3,3,\frac{3-1-2}{3}+1,0,\frac{3-2-1}{3}+2,\frac{3-1-2}{3},\frac{3-2-1}{3})=(5,3,3,1,0,2,0,0)$ PDPDS in $\mathbb{Z}_5\times\mathbb{Z}_3$. 
		Similarly, $\underline{a}=(0,0,\zeta_3^2,\zeta_3,1,\zeta_3,\zeta_3^2)$  is an almost 3-ary NPS of type (-2,0) and $R$ is an $(7,3,5,1,0,0,1,2)$ PDPDS in $\mathbb{Z}_7\times\mathbb{Z}_3$. 
	\end{example}
	
	Next we obtain a generalization of a well known theorem on difference sets, see \cite[Lemma VI.5.4]{beth1999design} or \cite[Proposition 1]{ozbudak2012nonexistence}.
	\begin{proposition}\label{count-s}
		Let $ R $ be a $(n+2,p,n,\frac{n-\gamma_2 - 2}{p}+\gamma_2,0,\frac{n-\gamma_1 -1}{p}+\gamma_1,\frac{n-\gamma_2 - 2}{p},\frac{n-\gamma_1 -1}{p})$ PDPDS in $G$ relative to $H$ and $N$. Let $R$ have $ s_i $ many elements having $ i $ in the second component for $ i=0,1,2,\dots,p-1 $. Then 
		\be \label{prob:si1}
		\sum_{j=0}^{p-1}{s_j}^2= (\frac{n-\gamma_2 - 2}{p}+\gamma_2)(n-1)+(\frac{n-\gamma_1 -1}{p}+\gamma_1)2+n
		\ee
		and \be \label{prob:si2} \sum_{j=0}^{p-1}s_js_{j-i}=(\frac{n-\gamma_2 - 2}{p})(n-1)+(\frac{n-\gamma_1 -1}{p})2
		\ee
		for each $ i=1,2,\dots, \lceil\frac{p-1}{2}\rceil$, where subscripts are computed modulo $ p $.
	\end{proposition}
	
	\begin{proof}
		Let $ \psi $ the map from $ G=H\times N $ to $ N $ sending $ (a,i) $ to $ i $. Let $ A $ be the multiset consisting of the images (counting multiplicities) of $ \psi $ restricted to $ R $. By reordering on $ A $ we have $$ A=\{*\underbrace{0,0,\dots,0}_{s_0},\underbrace{1,1,\dots,1}_{s_1},\underbrace{2,2,\dots,2}_{s_2},\dots,\underbrace{p-1,p-1,\dots,p-1}_{s_{p-1}}*\} .$$
		So,
		$$ s_0=|\{(b,i)\in R: i=0\}| ,\dots,s_{p-1}=|\{(b,i)\in R: i=p-1\}|$$ and 
		\be  \label{s sum}
		s_0+s_1+s_2+\dots+s_{p-1}=|R|=n.
		\ee
		Let $\mathcal{T}_i $ be the subset $ R\times R $ defined as $$\mathcal{T}_i=\{(\beta_1,\beta_2)\in R\times R: \beta_1\neq\beta_2 \: and \: \psi(\beta_1-\beta_2)=i\}. $$
		As $ R $ is a $(n+2,p,n,\frac{n-\gamma_2 - 2}{p}+\gamma_2,0,\frac{n-\gamma_1 -1}{p}+\gamma_1,\frac{n-\gamma_2 - 2}{p},\frac{n-\gamma_1 -1}{p})$ PDPDS, for the cardinality $ |\mathcal{T}_i| $ of $ \mathcal{T}_i $, using definition of PDPDS, we obtain that
		\be \label{diagram}
		{|\mathcal{T}_i| = \left\lbrace
			\begin{array}{ll}
				(\frac{n-\gamma_2 - 2}{p}+\gamma_2)(n-1)+(\frac{n-\gamma_1 -1}{p}+\gamma_1)2, & i=0 \\ 
				(\frac{n-\gamma_2 - 2}{p})(n-1)+(\frac{n-\gamma_1 -1}{p})2, & 0\lneq i \leq p-1\\
			\end{array}
			\right.}
		\ee 	
		Let define as $\mathcal{T}_{i,j}=\{(\beta_1,\beta_2)\in \mathcal{T}_i: \psi(\beta_1)=j\} \subset \mathcal{T}_i$  for $ 0\leq i \leq p-1 $ and $ 0\leq j \leq p-1 $. 
		Then we have, \be \label{eq} |\mathcal{T}_i| =\sum_{j=0}^{p-1}|\mathcal{T}_{i,j}|. \ee
		For $ 0\lneq i \leq p-1  $ and $ 0\leq j \leq p-1 $, we determine $ \mathcal{T}_{i,j} $. We know that $ (\beta_1,\beta_2)\in \mathcal{T}_{i,j} $ if only if $ \beta_1 \in R $, $ \psi(\beta_1)=j $ and $ \beta_2 \in R $, $ \psi(\beta_2)=j-i $. Therefore we get that , \\
		$$ |\{\beta_1\in R: \psi(\beta_1)=j\}|=s_j \; and \;  |\{\beta_2\in R: \psi(\beta_2)=j-i\}|=s_{j-i},$$
		where we define the subscript $ j-i $ modulo $ p $. Hence using  (\ref{diagram}) and (\ref{eq}) we conclude that 
		\be \label{eq2}
		(\frac{n-\gamma_2 - 2}{p})(n-1)+(\frac{n-\gamma_1 -1}{p})2=\sum_{j=0}^{p-1}s_js_{j-i} . 
		\ee
		Remark that it is enough to consider the subset of equation in (\ref{eq2}) corresponding to $ 1\leq i\leq \lceil\frac{p-1}{2}\rceil$ because each equation in (\ref{eq2}) with $ \lceil\frac{p-1}{2}\rceil \leq i\leq p-1 $ is the same as an equation in (\ref{eq2}) with $ 1\leq i\leq \lceil\frac{p-1}{2}\rceil$. \\
		For $ 0\leq j \leq p-1 $, we determine $ \mathcal{T}_{0,j} $. We know that $ (\beta_1,\beta_2)\in \mathcal{T}_{0,j} $ if only if $ \beta_1 \in R $ , $ \psi(\beta_1)=j $ and $ \beta_2 \in R $, $ \psi(\beta_2)=j $ and $ \beta_1 \neq \beta_2 $.  Therefore we get that $ |\mathcal{T}_{0,j}|=s_j(s_j-1) $ for $ 0\leq j \leq p-1 $. Hence using  (\ref{s sum}),(\ref{diagram}) and (\ref{eq}) we conclude that 
		\be \nn
		\begin{aligned}
			(\frac{n-\gamma_2 - 2}{p}+\gamma_2)(n-1)+(\frac{n-\gamma_1 -1}{p}+\gamma_1)2&=
			 \sum_{j=0}^{p-1}s_j(s_j-1)=\sum_{j=0}^{p-1}s_j^2 -n
		\end{aligned}
		 \ee
		and therefore
		$$ \sum_{j=0}^{p-1}{s_j}^2= (\frac{n-\gamma_2 - 2}{p}+\gamma_2)(n-1)+(\frac{n-\gamma_1 -1}{p}+\gamma_1)2+n.$$
	\end{proof}	\\
	
	Using Propositions \ref{pDPDS equality} and \ref{count-s}, we get the following result.
	\begin{corollary}\label{sum sj}
		Let $ R $ be a $(n+2,p,n,\lambda_1,\lambda_2,\lambda_3,\mu_1,\mu_2) = (n+2,p,n,\frac{n-\gamma_2 - 2}{p}+\gamma_2,0,\frac{n-\gamma_1 -1}{p}+\gamma_1,\frac{n-\gamma_2 - 2}{p},\frac{n-\gamma_1 -1}{p})$ PDPDS in $G$ relative to $H$ and $N$. Let $R$ have $ s_i $ many elements having $ i $ in the second component for $ i=0,1,2,\dots,p-1 $. Then \be \label{cor:si}
		(\sum_{j=0}^{p-1}s_js_{j-i})(p-1)+\sum_{j=0}^{p-1}{s_j}^2=n^2
		\ee
		for each $ i=1,2,\dots, \lceil\frac{p-1}{2}\rceil$ where subscripts are computed modulo $ p $.
	\end{corollary}
	
	\begin{proof}
		We multiply \eqref{prob:si1} by $ p-1 $ and add to \eqref{prob:si2}, so we get
		\be \nn
		\begin{aligned}
			(\sum_{j=0}^{p-1}s_js_{j-i})(p-1)+\sum_{j=0}^{p-1}{s_j}^2 
			&= ((\frac{n-\gamma_2 - 2}{p})(n-1)+(\frac{n-\gamma_1 -1}{p})2)(p-1) \\
			&+(\frac{n-\gamma_2 - 2}{p}+\gamma_2)(n-1)+(\frac{n-\gamma_1 -1}{p}+\gamma_1)2+n.
		\end{aligned}
		\ee
		Equivalently, we have
		\be \nn
		\begin{aligned}
		(\sum_{j=0}^{p-1}s_js_{j-i})(p-1)+\sum_{j=0}^{p-1}{s_j}^2&=
		(\mu_1(n-1)+\mu_22)(p-1)+\lambda_1(n-1)+\lambda_32+n\\
		&=(n-1)(\lambda_1+(p-1)\mu_1)+2(\lambda_3+(p-1)\mu_2)+n.
		\end{aligned}
		\ee
		Finally, by Proposition \ref{pDPDS equality}, we get the result
		\be \nn
			(\sum_{j=0}^{p-1}s_js_{j-i})(p-1)+\sum_{j=0}^{p-1}{s_j}^2&=
			n^2-n+n=n^2.
		\ee
	\end{proof}\\

	Below we prove a bound on $\gamma_2$ for the existence of a NPS of type $(\gamma_1,\gamma_2)$ by using Proposition \ref{prop:counting}.
	\begin{theorem}\label{thm:bound}
		Let $ p $ be an odd prime number, $ n\in \mathbb{Z}^+$, $\gamma_1, \gamma_2 \in \Z$ such that $n-\gamma_2 - 2 = k_1p$ and $n-\gamma_1-1 =k_2p$ for some $k_1,k_2 \in \N$. Then, there does not exist an almost $ p $-ary sequence of  type $ (\gamma_1,\gamma_2) $ and period $ n+2 $ with two consecutive zero-symbols for $\gamma_2 \le \left\lfloor\frac{-pk_1-4+ \sqrt{p^2k_1^2-4pk_1+8pk_2}}{2}\right\rfloor$.	
	\end{theorem}
	
	\begin{proof}
		Assume there exists an almost $p$-ary sequence of length $n$ and type $ (\gamma_1,\gamma_2) $ such that $\gamma_2 > \left\lfloor\frac{-pk_1-4+ \sqrt{p^2k_1^2-4pk_1+8pk_2}}{2} \right\rfloor$. 
		Set $n-\gamma_2-2=pk_1 $ and $ n-\gamma_1-1=pk_2 $ in (\ref{prob:si1}), and so we get
		\be \label{thm:si2-1}
		\sum_{j=0}^{p-1}{s_j}^2= pk_1\gamma_2+3pk_1+pk_1^2+\gamma_2k_1+k_1+2k_2-2pk_2+(\gamma_2+2)^2 .	
		\ee
		On the other hand we know that $s_0 + s_1 + \cdots + s_{p-1} = n=pk_1+\gamma_2+2$, that is,
		\be \label{thm:sum_si}
		\sum_{j=0}^{p-1}{s_j}=pk_1+\gamma_2+2.
		\ee 
		Hence \eqref{thm:sum_si} gives that 
		\be \label{thm:si2-2}
		\sum_{j=0}^{p-1}{s_j^2} \ge (\frac{pk_1+\gamma_2+2}{p})^2p
		=pk_1^2+2k_1(\gamma_2+2)+\frac{(\gamma_2+2)^2}{p}.
		\ee 
		We consider \eqref{thm:si2-1} and \eqref{thm:si2-2} together. And we get 
		$$pk_1^2+2k_1(\gamma_2+2)+\frac{(\gamma_2+2)^2}{p} \le pk_1\gamma_2 + 3pk_1 + pk_1^2  + \gamma_2k_1 + k_1 + 2k_2-2pk_2+(\gamma_2+2)^2.
		$$ 
		Equivalently, we have
		\be \label{eq:gamma2}(1-p)((\gamma_2+2)^2+pk_1(\gamma_2+2)+pk_1-2pk_2)\leq 0
		\ee
		Firstly, \eqref{eq:gamma2} holds if $ \gamma_2 \leq \frac{-pk_1-\sqrt{p^2k_1^2-4pk_1+8pk_2}}{2}-2 < -pk_1-1$. But this contradicts to $n-\gamma_2 -2 = pk_1$. Secondly, \eqref{eq:gamma2} holds if $ \gamma_2 \geq \frac{-pk_1+ \sqrt{p^2k_1^2-4pk_1+8pk_2}}{2}-2$, but this contradicts to the beginning assumption. 
	\end{proof}	
	
\begin{table}[t] 
	\caption{Non-existence results on NPS by Theorem \ref{thm:bound} for n=15}
\centering
\label{tab:gamma2}
\begin{tabular}{|c|c|c|c|}
            \hline
            \multicolumn{4}{|c|}{p=5}\\
			\hline
            $\gamma_1$ & $\gamma_2$ & B & Comments\\
            \hline
			-10 & -8 & -1 & not exist\\
			\hline
			-10 & -5 & -1 & not exist\\
			\hline
			-10 & -2 & -1 & not exist \\
			\hline
			-10 & 1 & 0 &  \\
			\hline
			-10 & 4 & 1 &   \\
			\hline
			-10 & 7 & 2 &   \\
			\hline
			-10 & 10 & 3 &  \\ 
			\hline
			-7 & -8 & -2 & not exist\\ 
			\hline
			-7 & -5 & -1 & not exist  \\
			\hline
			-7 & -2 & -1 & not exist \\ 
			\hline
			-7 & 1 & 0 &   \\
			\hline
			-7 & 7 & 1 &   \\
			\hline
			-7 & 10 & 2 &  \\ 
			\hline
			-4 & -8 & -2 & not exist\\ 
			\hline
			-4 & -5 & -2 & not exist \\
			\hline
			-4 & -2 & -1 & not exist \\
			\hline
\end{tabular}
\quad
\begin{tabular}{|c|c|c|c|}
            \hline
			\multicolumn{4}{|c|}{p=5}\\
			\hline
			$\gamma_1$ & $\gamma_2$ & B & Comments\\
            \hline
			-4 & 1 & -1 &  \\
			\hline
			-4 & 4 & 0 &  \\
			\hline
			-4 & 7 & 1 &  \\
			\hline
			-4 & 10 & 2 &  \\
			\hline
			-1 & -8 & -2 & not exist \\ 
			\hline
			-1 & -5 & -2 & not exist \\
			\hline
			-1 & -2 & -2 & not exist \\
			\hline
			-1 & 1 & -1 &  \\
			\hline
			-1 & 4 & -1 &  \\
			\hline
			-1 & 7 & 0 &  \\
			\hline
			-1 & 10 & 1 &  \\
			\hline
			2 & -8 & -2 & not exist\\  
			\hline
			2 & -5 & -2 & not exist  \\ 
			\hline
			2 & -2 & -2 & not exist \\ 
			\hline
			2 & 1 & -2 &   \\
			\hline
			2 & 4 & -1 &  \\
			\hline
\end{tabular}
\quad
\begin{tabular}{|c|c|c|c|}
            \hline
            \multicolumn{4}{|c|}{p=5}\\
			\hline
			$\gamma_1$ & $\gamma_2$ & B & Comments\\
            \hline
			2 & 7 & 0 &  \\
			\hline
			2 & 10 & 1 & \\
			\hline
			5 & -8 & -3 & not exist\\  
			\hline
			5 & -5 & -2 & not exist \\
			\hline
			5 & -2 & -2 & not exist \\
			\hline
			5 & 1 & -2 &  \\
			\hline
			5 & 4 & -2 &  \\
			\hline
			5 & 7 & -1 &  \\
			\hline
			5 & 1 & 0 &  \\
			\hline
			8 & -8 & -3 & not exist\\ 
			\hline
			8 & -5 & -3 & not exist \\
			\hline
			8 & -2 & -3 &  \\
			\hline
			8 & 1 & -2 & \\
			\hline
			8 & 4 & -2 &  \\
			\hline
			8 & 7 & -2 &  \\
			\hline
			8 & 10 & -1 &  \\
\hline
\end{tabular}
\quad
\begin{tabular}{|c|c|c|c|}
			\hline
			\multicolumn{4}{|c|}{p=3}\\
			\hline
			$\gamma_1$ & $\gamma_2$ & B & Comments\\
            \hline
			-6 & -7 & -2 &not exist  \\
			\hline
			-6 & -2 & -1 & not exist\\
			\hline
			-6 & 3 & 0 & \\  
			\hline
			-6 & 8 & 1 & \\
			\hline
			-1 & -7 & -2 & not exist \\
			\hline
			-1 & -2 & -2 &  not exist\\
			\hline
			-1 & 3 & -1 &  \\
			\hline
			-1 & 8 & 0 &  \\
			\hline
			4 & -7 & -2 & not exist \\
			\hline
			4 & -2 & -2 & not exist\\ 
			\hline
			4 & 3 & -2 &  \\
			\hline
			4 & 8 & 0 &  \\
			\hline
			9 & -7 & -3 & not exist\\
			\hline
			9 & -2 & -3 &  \\
			\hline
			9 & 3 & -2 &  \\
			\hline
			9 & 8 & -2 &  \\
\hline
\end{tabular}
\end{table}

	 We tabulate some nonexistence results obtained by Theorem \ref{thm:bound} for $n=15$, $-10 \leq \gamma_1, \gamma_2 \leq 10$, $p=5$ and $p=3$ respectively in Table \ref{tab:gamma2}, where $B$ is the upper bound on $\gamma_2$ given in Theorem \ref{thm:bound}. Pairs $(\gamma_1,\gamma_2)$ not included in Table \ref{tab:gamma2} are exclude by Corollary \ref{cor:NPS2}. The empty rows in the table are undecided cases. It is seen  that the case $\gamma_2 = -2$ and $B=-3$ appears in the table, but Theorem \ref{thm:bound} does not say anything about the status of its existence.
	 Actually, it is easily seen that the upper bound on $\gamma_2$ in Theorem \ref{thm:bound} is at least -3. Hence we have the following corollary.
	\begin{corollary}\label{cor:gamma2}
		Let $ p $ be an odd prime number and $ n\in\Z^+ $. Then there does not exist an almost $ p $-ary sequence of type $ (\gamma_1,\gamma_2) $ and period $ n+2 $ with two consecutive zero-symbols for  $\gamma_2 \leq -3$.
	\end{corollary}

	\section{Conclusion}
	In this paper, we prove a lower and an upper bounds on the number of  distinct out-of-phase autocorrelation coefficients of an almost $p$-ary sequence of period $n+s$ with $s$ consecutive zero-symbols.  Theorem \ref{th:nbrg} shows that  the number of distinct out-of-phase autocorrelation coefficients is between $\min\{s,p,n\}$ and $n-1+\min\{n,s\}$. Therefore one can not get an NPS of type $\gamma$ by adding extra zero-symbols at consecutive positions. We next prove in Theorem \ref{th:p-DPDS} that a $p$-ary NPS of type $(\gamma_1, \gamma_2)$ is equivalent to a PDPDS. Then, we obtain that they only exist when $p$ divides $n-\gamma_2-2$ and $n-\gamma_1-1$. We give in 
	Theorem \ref{thm:bound}  a necessary condition on  $ \gamma_2 $ for the existence of an almost $p$-ary NPS of type $(\gamma_1, \gamma_2)$. In particular we show that they don't exist for $ \gamma_2\leq-3 $.
	
\section*{Acknowledgement}
The authors are supported by the Scientific and Technological Research	Council of Turkey (TÜBİTAK) under Project No: \mbox{116R026}. The authors would like to thank Alexander Pott for constructive criticism of the manuscript.

    \bibliographystyle{splncs03}
	\bibliography{ozdenyayla}   
\end{document}